\documentclass{amsart}
\usepackage{amsmath, amsfonts, amssymb, euscript, graphicx, verbatim, amsthm,xy,xypic}

\newcommand{\cS}{\mathcal{S}}
\newcommand{\abs}[1]{\left\vert#1\right\vert}
\newcommand{\set}[1]{\left\{#1\right\}}

\newcommand{\vno}{\varnothing}
\newcommand{\one}{{1}}
\newcommand{\0}{{0}}
\newcommand{\sbs}{\subset}
\newcommand{\sbeq}{\subset}
\newcommand{\ra}{\rightarrow}
\newcommand{\ld}{\{1, \ldots, n\}}

\newcommand{\id}{\mathrm{id}}
\newcommand{\isd}{\mathcal{IS}_n}
\newcommand{\sd}{\mathcal{S}_n}
\newcommand{\isdd}{\mathcal{IS}_m}

\newcommand{\is}{\mathcal{IS}}
\newcommand{\nd}{\mathcal{N}_n}
\newcommand{\n}[1]{\mathcal{N}_{#1}}
\newcommand{\ndd}{\mathcal{N}_m}

\newcommand{\spx}{S^{PX}}
\newcommand{\pwr}{\,{\wre_p}\,}
\newcommand{\wisd}{\isdd \pwr \isd}
\newcommand{\wisdd}{\isdd \pwr \isd}
\newcommand{\rmi}[1]{R(\overrightarrow{M_{#1}})}
\newcommand{\vph}{\varphi}

\newcommand{\bv}[1]{\big|_{#1}}

\newcommand{\lc}{\mathbin\mathcal{L}}
\newcommand{\rc}{\mathbin\mathcal{R}}
\newcommand{\hc}{\mathbin\mathcal{H}}

\newcommand{\med}{\mathop{\vert}}

\newcommand{\al}{\alpha}

\DeclareMathOperator{\dom}{dom}
\DeclareMathOperator{\ran}{im}

\DeclareMathOperator{\aut}{Aut}
\DeclareMathOperator{\pa}{PAut}
\DeclareMathOperator{\wre}{\wr}

\newtheorem{theorem}{Theorem}
\newtheorem{proposition}{Proposition}
\newtheorem{lemma}{Lemma}
\newtheorem{corollary}{Corollary}
\theoremstyle{definition}
\newtheorem{definition}{Definition}
\theoremstyle{remark}
\newtheorem{remark}{Remark}

\begin{document}

\title[On isomorphisms of cross-sections
of $\wisdd$]{ON ISOMORPHISMS OF $\rc$- AND $\lc$-CROSS-SECTIONS\\ OF WREATH PRODUCTS\\ OF FINITE INVERSE SYMMETRIC SEMIGROUPS.}

\author{EUGENIA KOCHUBINSKA}
\address{Taras Shevchenko National University of Kyiv, Faculty of Mechanics and Mathematics, Volodymyrska str. 64, 01601, Kyiv, Ukraine.
}

%

\maketitle


\begin{abstract}

We classify $\rc$- and $\lc$-cross-sections of wreath
products of finite inverse symmetric semigroups $\wisdd$ up to
isomorphism. We show that every isomorphism of $\rc$ ($\lc$-)
cross-sections of $\wisdd$ is a conjugacy. As an auxiliary result,
we get that every isomorphism of $\rc$- ($\lc$-) cross-sections of
$\isd$ is also a conjugacy. We also compute  the number of
non-isomorphic $\rc$ ($\lc$-) cross-sections of $\wisdd$.
\end{abstract}


\section{Introduction}
Transformation semigroups play an important role in semigroup
theory. One of the reasons is that transformation semigroups appeared
in recent studies in general symmetry theory as  (full or partial)
endomorphisms semigroup of different combinatorial objects. 

The study of cross-sections of semigroups was started by
Renner~\cite{renner}. Later on different authors have studied
cross-sections of particular semigroups. $\hc$-cross-sections of
inverse symmetric semigroups were deeply studied by Cowan and
Reilly~\cite{CoReilly}, $\rc$- and $\lc$- cross-sections of $\isd$
were classified by Ganyushkin and Mazorchuk~\cite{GM}. $\rc$- and
$\hc$-cross-sections for the full finite transformation semigroup
$\mathcal{T}_n$   and for the infinite full transformation semigroup
$\mathcal{T}_X$ were classified by
Pyekhtyeryev~\cite{vasia1,vasia2}.

In the present paper we continue the study of $\rc$- and $\lc$-
cross-sections of partial wreath products of finite inverse symmetric
semigroups initiated in~\cite{endm}. We classify $\rc$- and $\lc$-
cross-sections of partial wreath products of finite inverse
semigroups up to isomorphism. The paper is organized as follows. All
necessary definitions are collected in Section~\ref{sec:basic}.
Section~\ref{sec:randl} contains known results on $\rc$- and
$\lc$-cross-sections of $\isd$ and $\isdd \pwr \isd$. Classification
of $\rc$- and $\lc$- cross-sections of $\isd$ and $\isdd \pwr \isd$
up to isomorphism is given in Section~\ref{sec:isomisn} and
Section~\ref{sec:isom} respectively.

\section{Basic definitions}\label{sec:basic}
For a set $X$, let $\mathcal{IS}(X)$ denote the set of all partial
bijections on $X$ with the natural composition law: $ f \circ g:
\dom(f) \cap f^{-1} (\dom(g)) \ni x \mapsto xfg$ for $f, \; g \in
\mathcal{IS}(X)$. The set $(\mathcal{IS}(X), \circ)$ is clearly an
inverse semigroup. This semigroup is called the \emph{full inverse
symmetric semigroup} on $X$. If $X=\nd$, where $\nd=\ld$, then
semigroup $\mathcal{IS}(\nd)$ is called the \emph{full inverse
symmetric semigroup of rank} $n$ and is denoted $\isd$. We
distinguish the element whose domain is $\vno$, it will be denoted
by $\0$. It is the zero of semigroup $\isd$. Also we distinguish an
identity map $\one:\isd \ra \isd$ defined by $x\one=x$ for all
$x\in\isd$. Clearly, this is the unity of $\isd$.

It is possible to introduce for elements of $\isd$ an analogue of
the cyclic decomposition for elements of the symmetric group
$\mathcal{S}_n$. We start with introducing two classes of elements.
Let $A=\{x_1, x_2, \ldots, x_k\} \sbs \nd$ be a subset. Denote by
$(x_1, x_2, \ldots, x_k)$ the unique element $f\in \isd$ such that
$x_i f=x_{i+1}$, $i=1,2, \ldots, k-1$, $x_k f=x_1$ and $xf=x, x
\notin A$. Assume that $A\neq \vno$ and denote by $[x_1, x_2,
\ldots, x_k]$ the unique element $f\in \isd$ such that
$\dom(f)=\nd\setminus\{x_k\}$ and $x_if=x_{i+1}$, $i=1,2, \ldots,
k-1$, and $xf=x$, $x \notin A$. The element $(x_1, x_2, \ldots,
x_k)$ is called a cycle and the element $[x_1, x_2, \ldots, x_k]$ is
called a chain, and the set $A$ is called the support of $(x_1, x_2,
\ldots, x_k)$ or $[x_1, x_2, \ldots, x_k]$. Any element of $\isd$
decomposes uniquely into a product of cycles and chains with
disjoint supports. This decomposition is called a \emph{chain
decomposition} \cite{GM1}. Denote by $\langle x_1, x_2, \ldots, x_k
\rangle$ the element $f\in \isd$ such that
$\dom(f)=A\setminus\set{x_k}$ and $x_i f=x_{i+1}$, $i=1,2, \ldots,
k-1$.

Recall the definition of a partial wreath product of semigroups. Let
$S$ be a semigroup, $(X,P)$ be the semigroup of partial
transformations of a set $X$. Define the set $\spx$ as the set of
partial functions from $X$ to $S$:
\begin{equation*}
\spx=\{f:A \ra S| \dom(f)=A, A \sbeq X\}.
\end{equation*}
Given $f,g \in \spx$, the product $fg$ is defined as:
\begin{equation*}
\dom(fg)=\dom(f)\cap\dom(g), (fg)(x)=f(x)g(x) \text{\ for all } x\in
\dom(fg).
\end{equation*}
For $a\in P, f\in \spx$, define $f^a$ as:
\begin{equation*}
\begin{gathered}
(f^a)(x)=f(xa),\ \dom(f^a)=\{x \in \dom(a); xa\in \dom(f)\}.
\end{gathered}
\end{equation*}

\begin{definition}
{The partial wreath product} of the semigroup $S$ with the semigroup
$(X,P)$ of partial transformations of the set $X$ is the set
$$\{(f,a)\in \spx \times (X,P)\,|\,\dom(f)=\dom(a) \}$$ with product defined by
$(f,a)\cdot (g,b)=(fg^a,ab).$ We will denote the partial wreath
product of semigroups $S$ and $(X,P)$ by $S \pwr P$.
\end{definition}
\begin{remark}
Some authors use the term wreath product. We follow terminology from the book of J.D.P Meldrum \cite{Meldrum}, where this construction is called partial wreath product.
\end{remark}

It is known \cite{Meldrum} that a partial wreath product of
semigroups is a semigroup. Moreover, a partial wreath product of
inverse semigroups is an inverse semigroup. An important example of
an inverse semigroup is the semigroup $\pa T_n^k$ of partial
automorphisms of a $k$-level $n$-regular rooted tree $T_n^k$. Here
by a partial automorphism we mean a root-preserving tree
homomorphism defined on a rooted subtree  of $T_n^k$. It is shown in
\cite{comb} that
$$\pa T_n^k \simeq \underset{k}{\underbrace{\isd \pwr \isd \pwr
\cdots \pwr \isd}}.$$ This is an analogue of the well-known fact
that $\aut T_n^k \simeq \mathcal{S}_n\wr\dots \wr \mathcal{S}_n$.

\begin{remark}
Let $\nd'=\nd \cup\{\emptyset\}$, and we can consider semigroup $\isd$ as a subsemigroup $K_{n+1}$ of full transformation semigroup $\mathcal{T}_{n+1}$ acting on the set $\nd'$: for every $a\in \isd$ we define $a'\in K_{n+1}$ as $xa'=xa$ for $x\in \dom(a)$ and $xa'=\emptyset$ for $x\notin\dom(a)$.
Let $S=K_{m+1} \wr K_{n+1}$, where $\wr$ is a wreath product of semigroups (see, e.g., \cite[Chapter 1]{eilenberg}, \cite[Chapter 10]{Meldrum}). We define equivalence $\sim$ on the semigroup $S$: $(f,a) \sim (g,a) \Leftrightarrow f(x)=g(x)$ for $x\in \nd'$ such that $xa\neq \emptyset$. Then we identify every element $(f,a)\in \isdd\pwr\isd$ with an element $(f',a')\in S/\sim$ in a following way. We set $xa'=xa$ for $x\in \dom(a)$ and $xa=\emptyset$ otherwise, and $f'(x)=f(x)$ for $x\in \dom(a)$, $f'(x)$ for $x\notin \dom(a)$ can be chosen arbitrarily. So we can consider partial wreath product $\isdd \pwr \isd$ as a quotient $S/\sim$.
\end{remark}

\section{$\rc$- and $\lc$-cross-sections of $\wisdd$}\label{sec:randl}
In this section we give a brief description of known results on
$\rc$- and $\lc$-cross-sections of the semigroup $\wisdd$.

Let $S$ be an inverse semigroup with identity. Recall that Green's
$\rc$-relation on inverse semigroup $S$ is defined by $a\rc b
\Leftrightarrow aS=bS$, similarly, Green's $\lc$-relation is defined
by $a\lc b \Leftrightarrow Sa=Sb$. It is well-known (e.g.
\cite{GM1}) that Green's relations on $\isd$ can be described as
follows: $a \rc b \Leftrightarrow \dom(a)=\dom(b)$; $a \lc b
\Leftrightarrow \ran(a)=\ran(b)$.

$\rc$- and $\lc$-relations on  $\wisdd$ are described in
\begin{proposition}{\em\cite{comb}} \label{green_2}
Let $(f,a)$, $(g,b) \in \wisdd$. Then
\begin{enumerate}
\item $(f,a)$ $\rc$ $(g,b)$ if and only if
$\dom(a)=\dom(b)$ and for any $z \in \dom(a)$\ \  $f(z)\rc
g(z)$;

\item $(f,a)\lc(g,b)$ if and only if
$\ran(a)=\ran(b)$ and for any $z \in \ran(a)$\ \
$g^{b^{-1}}(z)\lc f^{a^{-1}}(z)$, where $a^{-1}$ is the inverse
for $a$.
\end{enumerate}
\end{proposition}

Now let $\rho$ be an equivalence relation on $S$. A subsemigroup $T
\sbs S$ is called \emph{cross-section with respect to} $\rho$ (or
simply $\rho$-cross-section) provided that $T$ contains exactly one
element from every equivalence class. Correspondingly, cross-sections
with respect to $\rc$- $(\lc$-) Green's relations are called $\rc$-
($\lc$-) cross-sections. Note that every $\rc$- ($\lc$-) equivalence
class of inverse semigroup contains exactly one idempotent. Then the number of elements
in every $\rc$- ($\lc$-) cross-section of inverse semigroup $S$ is $\abs{E(S)}$, where $E(S)$ is the
subsemigroup of all idempotents of semigroup $S$.

Observe that a subsemigroup $H$ of semigroup $\isd$ is an
$\rc$-cross-section if and only if for every subset $A\subseteq \nd$
it contains exactly one element $a$ such that $\dom(a)~=~A$.

Before describing $\rc$- and $\lc$-cross-sections in semigroup
$\wisdd$, we give first the description of $\rc$- and
$\lc$-cross-sections in semigroup $\isd$ presented in \cite{GM}.
Let $\nd=M_1 \sqcup M_2 \sqcup \ldots \sqcup M_s$~be an arbitrary
decomposition of $\nd=\{1,2, \ldots, n\}$ into disjoint union of
non-empty blocks, where the order of blocks is irrelevant. Assume
that a linear order is fixed on the elements of every block: $M_i =
\set{m_1^i<m_2^i<\dots<m_{\abs{M_i}}^i}$.

For each pair $i,j$,\  $1 \leq i \leq s$, $1 \leq j \leq |M_i|$
define $a_{i,j}=[m_1^i, m_2^i, \ldots, m_j^i]$ and denote by
$R(\overrightarrow{M_1}, \overrightarrow{M_2}, \ldots,
\overrightarrow{M_s})=\langle a_{i,j} |\; 1 \leq i \leq s, 1\leq j
\leq |M_i|\rangle \sqcup \{e\}$.

\begin{theorem}{\em\cite{GM}}
For an arbitrary decomposition  $\nd=M_1 \sqcup M_2 \sqcup \ldots \sqcup
M_s$ and arbitrary linear orders on the elements of every block of
this decomposition the semigroup $R(\overrightarrow{M_1},
\overrightarrow{M_2}, \ldots, \overrightarrow{M_s})$ is an
$\rc$-cross-section of $\isd$. Moreover, every $\rc$-cross-section
is of the form  $R(\overrightarrow{M_1}, \overrightarrow{M_2},
\ldots, \overrightarrow{M_s})$ for some decomposition $\nd=M_1
\sqcup M_2\sqcup \ldots \sqcup M_s$ and some linear orders on the
elements of every block.
\end{theorem}

Since the map $a \mapsto a^{-1}$ is an anti-isomorphism of the
semigroup $\isd$, which sends  $\rc$-classes to $\lc$-classes, then
$\lc$-cross-sections are described similarly.

Now we turn to the description of $\rc$- and $\lc$-cross-sections of
semigroup $\wisd$. It follows from Proposition \ref{green_2} that a
subsemigroup $H \sbs\wisd$ is an $\rc$-cross-section if and only if
for any $A \sbs \nd$ and any collection of sets $B_{x_1}, \ldots,
B_{x_{|A|}}\sbs \ndd$ there exists exactly one element $(f,a)\in H$
satisfying $\dom(a)=A$ and $\dom(x_i f)=B_{x_i}$ for all $x_i \in
A$. Later on we will use  this fact frequently.

Define the  map $\vph_{\mu}:\prod_{i=1}^k(S\pwr \is{(M_i)}) \ra
S\pwr \isd$ in the following manner: $\vph_{\mu}$ maps the product
$\prod_{i=1}^k(f_i,a_i)$ to the element $(f,a)$ such that
$a\bv{M_i}=a_i$, $f\bv{M_i}=f_i$.
\begin{theorem}{\em\cite{endm}}
\label{thm:rcross} Let $R(\overrightarrow{M_1},
\overrightarrow{M_2}, \ldots, \overrightarrow{M_k})$ be an
$\rc$-cross-section of semigroup $\isd$,\break $R_1, \ldots$, $R_k$
be $\rc$-cross-sections of semigroup $\isdd$. Then
$$R=\vph_{\mu}\left((R_1\pwr R(\overrightarrow{M_1}) )\times (R_2 \pwr
R(\overrightarrow{M_2}) ) \times \ldots \times (R_{k}\pwr
R(\overrightarrow{M_k}) )\right)$$  is an $\rc$-cross-section of semigroup $\isdd
\pwr \isd$.

Moreover, every  $\rc$-cross-section of semigroup $\isdd \pwr \isd$
is isomorphic to
$$(R_1 \pwr R(\overrightarrow{M_1}) )\times (R_2 \pwr
R(\overrightarrow{M_2}) ) \times \ldots \times (R_{k}\pwr
R(\overrightarrow{M_k}).$$
\end{theorem}

A map $(f,a) \mapsto (f,a)^{-1}$ is an anti-isomorphism of semigroup
$\wisdd$, that sends $\rc$-classes to $\lc$-classes. It is also
clear that it maps $\rc$-cross-sections to $\lc$-cross-sections and
vice-versa. Hence dualizing Theorem \ref{thm:rcross}, one gets a
description of $\lc$-cross-sections.

\section{Isomorphisms of $\rc$- and  $\lc$-cross-sections of
$\isd$}\label{sec:isomisn}

Clearly, it is enough to study problem of isomorphism only for
$\rc$-cross-sections. The result for $\lc$-cross-sections is
analogous.

For an arbitrary  $\rc$-cross-section
$R=R(\overrightarrow{M_1},\overrightarrow{M_2},\ldots,
\overrightarrow{M_s})$ we call an idempotent $e\in E(R)$ \emph{block
idempotent} if $\dom(e) \sbs M_i$ for some $i$. The idempotents of
$R$ are described below.
Let $$R(\overrightarrow{M_i})=\set{a_{i,j} \med 1\leq j \leq
\abs{M_i}}\cup {e}.$$ Then $$R(\overrightarrow{M_1},
\overrightarrow{M_2}, \ldots, \overrightarrow{M_s}) \simeq
R(\overrightarrow{M_1})\times \ldots \times
R(\overrightarrow{M_s}).$$ Isomorphism is established by
$\vph(a)=(a\bv{M_1}, \ldots, a\bv{M_s})$.

For an $\rc$-cross-section $R=R(\overrightarrow{M_1},
\overrightarrow{M_2}, \ldots, \overrightarrow{M_s})$ element $e\in
R$ is an idempotent iff $e \bv{M_i} \in E(\rmi{i})$ for all $i=1,
\ldots, s$. An element $a\in \rmi{i}$ with domain
$\dom(a)=\set{m_{j_1}^i, \ldots, m_{j_k}^i\med j_1<\ldots < j_k}$
acts in a following way: $(m_{j_l}^i)a=m^i_{\abs{M_i}-k+l}$. Taking
it into account we get that idempotents of $\rmi{i}$ are described
as $e=(m_j^i)(m_{j+1}^i)\ldots(m^i_{\abs{M_i}})$.

Recall that on the set of idempotents the  partial order $\preceq$
is defined as $e \preceq f \Leftrightarrow ef=fe=e$.
\begin{proposition} \label{prop:blok}
An idempotent  $e \in E(R)$ is a block idempotent if and only if
there exist no idempotents $e_1 \neq 0$, $e_2 \neq 0$ such that $e_1
\preceq e$, $e_2 \preceq e$, $e_1 e_2 = 0$.
\end{proposition}
\begin{proof}
\emph{Necessity.} Let $e$ be block idempotent. Let $e_1 \preceq e$,
$e_2 \preceq e$ be idempotents of $R$ such that  $e_1 \neq 0$, $e_2
\neq 0$.  If $M_i=\set{x_1, \ldots, x_k\mid x_1<x_2<\ldots<x_k}$,
then for an arbitrary block idempotent $e$ it holds $\dom(e)=\{x_j,
\ldots, x_k\}$, $j \leq k$. Since $e, e_1, e_2 \in E(R)$, then $e_1
e_2 \neq 0$.

\emph{Sufficiency.} Assume the contrary. It means that there is no
index $i$ such  that $\dom(e)\sbs M_i$. Consider a block $M_i$ for
which $\dom(e) \cap M_i \neq \vno$. Then $e_1=e\bv{M_i}$ and
$e_2=e\bv{\overline{M_i}}$ are idempotents. Evidently, $e_1e_2=0$.
From condition $\dom(e) \cap M_i \neq \vno$ it follows $e_1\neq 0$,
and from the fact that $e$ is not block idempotent, it follows
$e_2\neq 0.$

\end{proof}

\begin{lemma}\label{lem:block}
Let $R_1, R_2$ be $\rc$-cross-sections of the semigroup $\isd$,
$\vph\colon R_1 \ra R_2$ be an isomorphism. Then there exists a
permutation $\Theta \in \sd$ such that for any block idempotent
$e\in R_1$ and every $x \in \{1, \ldots, n\}$ it holds
$x e\Theta=x\Theta\vph(e)$.
\end{lemma}
\begin{proof}
Let $R_1=R(\overrightarrow{M_1},\overrightarrow{M_2},\ldots,
\overrightarrow{M_s})$. The domains of block idempotents defined on
different blocks are disjoint. Thus if $i\neq j$ and $\dom(e_i)\sbs
M_i$, $\dom(e_j)\sbs M_j$, then $e_i e_j =0$. From the definition of
the partial order, we have for block idempotents defined on the same
block that $e_i \preceq e_j$ if and only if $\dom(e_i)\sbs
\dom(e_j)$. Then the set $E_b(R)$ of block idempotents of
$\rc$-cross-section $R$ as a poset can be drawn as:
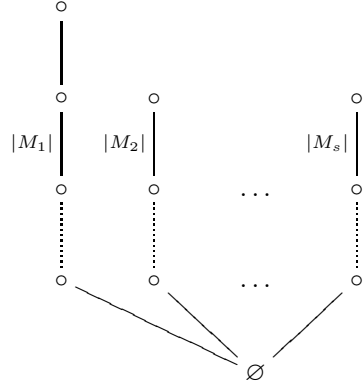
\begin{center}
\leavevmode
\xymatrix {
\circ & \\
\circ\ar@{-}[u] & \circ &\ & \circ\\
\circ\ar@{-}[u]^{|M_1|} & \circ\ar@{-}[u]^{|M_2|} & \ldots  & \circ\ar@{-}[u]^{|M_s|}\\
\circ\ar@{.}[u] & \circ\ar@{.}[u] &\ldots & \circ\ar@{.}[u]\\
\ &\ & \vno \ar@{-}[ur]\ar@{-}[ul]\ar@{-}[ull]&\ &\ }
\end{center}

It follows from Proposition \ref{prop:blok} that the property to be
a block idempotent is preserved under an isomorphism. On the other
hand, the isomorphism $\vph$ preserves the order $\preceq$. Thus
$\vph$ defines a poset isomorphism between $E_b(R_1)$ and
$E_b(R_2)$. If $M_i=\set{x_1<\ldots < x_k}$, then for a block
idempotent $e\in R_1$ with $\dom(e)=\{x_j, \ldots, x_k\}$ we put
$\nu_1(e)=x_j$. We define $\nu_2(e)$ for a block idempotent $e\in
R_2$ similarly. It is easily checked that $\Theta$ defined by
$i\Theta=\vph(i \nu_1^{-1})\nu_2^{\vphantom{1}}, i\in\{1,
\ldots, n\}$, is as required.
\end{proof}

\begin{theorem}\label{thm:isomisn}
Let $R_1, R_2$ be $\rc$-cross-sections of the semigroup $\isd$,
$\vph\colon R_1 \ra R_2$ be an isomorphism. Then there exists an
element $\Theta \in \sd$ such that for any $\alpha \in R_1$ and $x
\in \{1, \ldots, n\}$ the following equality holds
$x \alpha \Theta= x\Theta\vph(\alpha)$.
\end{theorem}
\begin{proof}
Let $\Theta$ be the permutation provided by Lemma \ref{lem:block}.

Let $\xi \in R_1=R(\overrightarrow{M_1},\overrightarrow{M_2},\ldots,
\overrightarrow{M_s})$ be such that $\dom(\xi)\sbs M_i$ and
$\ran(\xi)\sbs M_i$ for some $i$. It means $\id_{M_i}\xi=\xi$, i.e.
$\xi$ acts inside the block $M_i$. Thus,
$$\vph(\xi)=\vph(\id_{M_i}\xi)=\vph(\id_{M_i})\vph(\xi)=\id_{\Theta(M_i)}\vph(\xi)$$
Therefore, $\zeta=\vph(\xi)$ acts inside the block $\Theta(M_i)$.

Let $$M_i=\{x_1, x_2, \ldots, x_{k_i} \med x_1<x_2<\ldots
<x_{k_i}\},$$ $$\Theta(M_i)=\{y_1, y_2, \ldots, y_{k_i} \med
y_1<y_2<\ldots <y_{k_i}\},$$ where $y_i=x_i\Theta$, $i=1,\ldots,
k_i$. Let $\dom(\xi)=\{x_{i_1}, \ldots, x_{i_l}\}, i_1<\ldots< i_l$,
$\dom(\zeta)=\{y_{j_1}, \ldots, y_{j_p}\}, j_1<\ldots<j_p$.

Then $x_{i_m}\xi=x_{i_{\abs{M_i}}-l+m}$, $m=1, \ldots, l$,
$y_{j_m}\zeta=y_{j_{\abs{M_i}}-p+m}$, $m=1, \ldots, p$. Denote
$e_m=\id_{\{x_m, \ldots, x_{\abs{M_i}}\}}$, $f_m=\id_{\{y_m, \ldots,
y_{\abs{M_i}}\}}$. We have then
\begin{equation} \label{e:eidemp}
e_{\abs{M_i}}\xi=\0,\ldots,  e_{i_l+1}\xi=\0,\  e_{i_l}\xi \neq \0,
\end{equation}
and
\begin{equation} \label{e:fidemp}
f_{\abs{M_i}}\zeta=\0,\  \ldots,\  f_{j_p+1}\zeta=\0,
f_{j_p}\zeta \neq \0.
\end{equation}

Applying to (\ref{e:eidemp}) the isomorphism $\vph$ and using Lemma
\ref{lem:block}, we get
\begin{equation} \label{e:ffidemp}
f_{\abs{M_i}}\zeta=\0,\ \ldots, f_{i_l+1}\zeta=\0,\
f_{i_l}\zeta \neq \0.
\end{equation}
It follows from equalities (\ref{e:fidemp}) and (\ref{e:ffidemp})
that $l=p$ and $i_l=j_p$.

Further, we have
\begin{gather*}
e_{i_l}\xi=e_{i_l-1}\xi=\ldots=e_{i_{l-1}-1}\xi\neq e_{i_{l-1}}\xi, \\
f_{i_l}\zeta=f_{i_l-1}\zeta=\ldots=f_{j_{l-1}-1}\zeta\neq
f_{j_{l-1}}\zeta.
\end{gather*}
Similarly, we get $j_{l-1}=i_{l-1}$. By induction we obtain
$i_m=j_m$ and $l=p$. Then for any $x \in \dom(\xi)$ it holds:
$x\xi\Theta=x\Theta\vph(\xi)$, because
$y_{i_m}\zeta=x_{i_m} \xi \Theta$.

We will show now that equality $x\alpha\Theta=x\Theta\vph(\al)$
is true for any $\al \in R_1$.

Let $\al \in R_1$, $x\in M_i$. Then
\begin{gather*}
x \alpha\Theta=x\al\bv{M_i}\Theta=x\,\id_{M_i}\,\al\Theta=x\Theta\,\vph(\id_{M_i}\,\al)
=\\=x\Theta(\vph(\id_{M_i})\,\vph(\al))=x\Theta(\id_{\Theta(M_i)}\,\vph(\al))=x\Theta\,\vph(\al).\end{gather*}

Thus, for every $\alpha \in R_1$,  $x\in\{1, \ldots, n\}$ it is true
that $x\al\Theta=x\Theta\,\vph(\al)$
\end{proof}
\begin{remark}
It is proved in \cite{GM} that two $\rc$- ($\lc$-) cross-sections in
$\isd$ are isomorphic  if and only if they are conjugate. Theorem
\ref{thm:isomisn} strengthens this result: \emph{every} isomorphism
of $\rc$- ($\lc$-) cross-sections is a conjugacy.
\end{remark}

\section{Isomorphisms of $\rc$- and $\lc$-cross-sections of
$\isdd\pwr\isd$}\label{sec:isom}

In this section we generalize the result of the previous section
that every isomorphism of $\rc$- ($\lc$-) cross-sections of inverse
symmetric semigroup $\isd$ is a conjugacy, to partial wreath products
of finite inverse symmetric semigroups.

\begin{theorem}\label{thm:perhapsthelast}
Let $R'$, $R''$~be $\rc$-cross-sections of the semigroup
$\isdd\pwr\isd$, $\vph\colon R'\to R''$ be an isomorphism. Then
there exists such an element  $\Theta = (\vartheta,\theta)\in
\cS_m\wre \cS_n$ that
$$
\vph\big((f,a)\big) = \Theta^{-1} (f,a) \Theta.
$$
In other words, if $(f,a)\in R'$ and $(g,b)=\vph\big((f,a)\big)$,
then $\dom b = \theta\big(\dom(a)\big)$ and for any $x\in\dom(a)$
\begin{equation*}
x a \theta= x\theta b, \quad
g\big(\theta(x)\big) = \vartheta^{-1}(x) f(x)
\vartheta(xa).
\end{equation*}
\end{theorem}

\begin{proof}
\emph{Step 1.} For an idempotent $e=(f_e,a_e)\in \isdd\pwr\isd$
denote
$$N_e = \set{\zeta=(f,a)\in \isdd\pwr\isd\mid e\zeta = \zeta}.$$
If $\zeta=(f,a)\in N_e$, then $\dom(a)\subset \dom(a_e)$ and
$\dom(xf)\subset \dom(x f_e)$ for any $x\in \dom(a)$. But every
element of $\rc$-cross-section is completely defined by the sets
$\dom(a)$ and $\dom(xf)$, $x\in\dom(a)$. Thus the number of
elements of
 $\rc$-cross-sections $R$, which are in $N_e$, is equal to
\begin{align*}
\abs{N_e\cap R}=\sum_{A\subset \dom(a_e)}\ \prod_{x\in A}
2^{\abs{\dom(x f_e)}} = \prod_{x\in\dom(a_e)}
\big(1+2^{\abs{\dom(x f_e)}}\big).
\end{align*}

Take an element $e' = (\0,\id_{\nd})$.  We have
 $e'\in R'$.

Set $R'\cap N_{e'}$ contains $2^n$ elements (they are elements
$(\0,a)$). Since $\vph$ is an isomorphism, then
$e''=(f'',a'')=\vph(e')$ is an idempotent,  and $\abs{R''\cap
N_{e''}}=2^n$. Therefore, we have $$ \prod_{x\in\dom(a'')}
\big(1+2^{\abs{\dom(f''(x))}}\big)=2^n.
$$
This clearly implies $\abs{\dom(f''(x))}=0$ for all $x\in\dom(a'')$ and
$\abs{\dom(a'')}=n$, which means $e''=e'$.

\emph{Step 2.} Define
$$
R'_1 = \set{a\in \isd\mid (f,a)\in R'}
$$
and similarly $R''_1$. Both $R_1'$ and $R_1''$ are
$\rc$-cross-sections of $\isd$. For every element $(f,a)\in
R$ the product $$(f,a)(\0, e)=(\0, e)(f,a)=(\0, a)\in R.$$ We have then $
(\0,R'_1) = e'R'$ and
$
(\0,R''_1) = e' R''$. Since $\vph$ is an isomorphism, then
$\vph\big((\0,R'_1)\big) = (\0,R''_1)$ and a map $\vph_1\colon
R'_1\to R''_1$, which is defined as $(\0,\vph_1(x)) = \vph((\0,x))$,
is an isomorphism. From Theorem~\ref{thm:isomisn} it follows that
there exists an element $\theta\in \cS_n$ such that $\vph_1(a) =
\theta^{-1} a \theta$.

Take an arbitrary element $(f,a)\in R'$ and put
$(g,b)=\vph\big((f,a)\big)$. We get
$$
(\0,\vph_1(a))=\vph\big((\0,a)\big)=\vph\big(e'(f,a)\big)=
e'\vph\big((f,a)\big) = e'(g,b) = (\0,b),
$$
which implies $b=\vph_1(a) = \theta^{-1} a \theta$. Renumbering
elements of the set $\nd$ in a proper way, we could obtain
$R'_1=R''_1 = R(M_1)\times\dots\times R(M_s)$ and $\vph_1 = \id$.
Then, evidently, $b=a$.

\emph{Step 3.} Define the ``maximal'' block idempotents:
$e_i=(\one_{M_i},\id_{M_i})$, where $\one_{M_i}(x) = \id_{\ndd}$,
$x\in M_i$. According to Lemma \ref{lem:block} maximal idempotents
are preserved under isomorphism. As $\vph$ acts identically on the
second component, we have $\vph(e_i) = e_i$. Thus
\begin{equation}\label{eq:first}
\vph\big((f|_{M_i},a|_{M_i})\big)= \vph\big(e_i(f,a)\big)=e_i
\vph\big((f,a)\big) =e_i (g,a) = (g|_{M_i},a|_{M_i}).
\end{equation}

Since $M_i^R=M_i$ and $\nd=M_1\sqcup M_2 \sqcup \ldots \sqcup M_s$,
then there exists a monomorphism from $R$ to $\prod_{i=1}^s(\isdd
\pwr \is(M_i))$ defined in the following way: $$(f,a) \mapsto
\left((f\bv{M_1}, a\bv{M_1}), \ldots,
(f\bv{M_s},a\bv{M_s})\right).$$

Using the monomorphism (defined as above) from $\rc$-cross-section
to Cartesian product of $\rc$-cross-sections of $\isdd\pwr \is(M_i)$
and equality (\ref{eq:first}),  we have that it is enough to prove
the proposition for ``restriction $\vph$ to cofactors''. That is we
may suppose that $R'_1 = R''_1 = R(M)$, and that the isomorphism
$\vph$ is identical on the second component (because evidently, this
property is preserved under restriction). Without loss of generality
assume that  $M=\set{1,\dots,l}$ with natural order.

Let we have some $\rc$-cross-section $R$. Let $(f_i, a_i)$, $i=1,
\ldots, l$, be  elements of $\rc$-cross-section $R$ such that
$$a_i=\langle 1,i, i+1, \ldots, l-1, l \rangle, \dom(f(1))=\ndd.$$
Put $\vph_i=f_i(1)$ for an element $(f_i,a_i) \in R$. Consider now
the map $\Theta\colon  \isdd \pwr\is(M) \ra \isdd \pwr\is(M)$, which
acts as follows: $(f,a) \mapsto (g,a)$, where for $x \in \dom(a)$
$g(x)=\vph_x f(x) \vph^{-1}_{x^a}$. It is easy to check that this
map is an isomorphism and an isomorphic image of $\rc$-cross-section
is an $\rc$-cross-section. In such a way we define the maps
$\Theta'$ for $R'$ and $\Theta''$ for $R''$.

Therefore, applying to both cross-sections $R'$ and $R''$ recently
defined maps $\Theta'$ and $\Theta''$, we obtain
$\rc$-cross-sections $\Theta'(R')=R'_2\pwr R(M)$ and
$\Theta''(R'')=R''_2\pwr R(M)$ (see Lemma 3.6 in \cite{endm}). Since
both $\Theta'$ and $\Theta''$ are clearly conjugacies and act
identically on the second component, then we may assume $R'=R'_2\pwr
R(M)$, $R''=R''_2\pwr R(M)$, and isomorphism $\vph$ acts identically
on the second component.

\emph{Step 4.} Consider the set $R_2 = \set{(f,\id_{\set{l}})\in
R'}$. Evidently, $R_2\simeq R_2'$ (the isomorphism is defined by
$\psi\big((f,\id_{\set{l}})\big)=f(l)$). Since $\vph$ is identical
on the second component, then $\vph(R_2) = \set{(f,\id_{\set{l}})\in
R''}\simeq R_2''$, hence $R_2''\simeq R_2'$. Moreover, as $\sigma =
\psi^{-1} \vph \psi$ is an isomorphism of $\rc$-cross-sections of
$\isdd$, then there exists $\theta_0\in\cS_m$ such that $\sigma(a) =
\theta^{-1}_0 a \theta_0^{\vphantom{-1}}$. Put $x\vartheta =
\theta_0$, $x\in M$. We will show that $\Theta=(\vartheta,\id_M)\in
\cS_m\wre\cS_l$ is as required.

For $j\in M$ denote $$\tau_j=
                                  \begin{cases}
                                    (f,\langle j,l \rangle)\in R', \dom(f(j))=\n{m}, & \text{ if } j<l; \\
                                    (f, \id_{\{l\}})\in R', \dom(f(l))=\ndd, & \text{ if } j=l.
                                  \end{cases}
                             $$

As $f(j)\in R_2'$, then $f(j) = \id_{\ndd}$. We claim that
$\vph(\tau_j) = \tau_j$. Indeed, $\tau_j$ is a unique element of
$R'$ of the form $(f,\langle j,l \rangle)$, which cannot be
represented as a product of an element of $R'$ of such a form and a
non-idempotent element from $R_2$. Since $\vph$ is identical on the
second component, then the set of the elements of the form
$(f,\langle j,l \rangle)$ is preserved under the action of $\vph$,
and the same is true for $R_2$. Moreover, $\vph$ preserves the
operation and idempotents, hence $\vph$ should preserve an element
$\tau_j$ also.

Further, for arbitrary $i,j\in M, i<j$, denote by
$\Lambda(i,j):=\{i,j+1,$ $j+2,\dots,l\}$ and consider an element
$\lambda_{ij}=(f,a)\in R'$ such that $\dom(a)=\Lambda(i,j)$ (i.e.
$ia = j$, $ka=k$ for $k>j$) and $\dom(xf) = \ndd$ (i.e.
$f(x)=\id_{\ndd}$) for $x\in\Lambda(i,j)$. We claim that
$\vph({\lambda_{ij}}) = \lambda_{ij}$. Indeed, if we denote
$\Lambda(j):= \set{j,j+1,\dots,l}$, then element $\lambda_{ij}$ can
be characterized by the property: it is the only element of $R'$ of
the form $(f,\id_{\Lambda(i,j)})$, which cannot be represented as
the product of the element of such a form and non-idempotent element
of $R'$ of the form $(f,\id_{\Lambda(j)})$. As this property is
preserved under $\vph$, we have $\vph({\lambda_{ij}}) =
\lambda_{ij}$.

Consider now an arbitrary element $(f,a)\in R'$. Let $(g,a) =
\vph\big((f,a)\big)$. Take any $i \in \dom(a)$ and let $ia=j$.
Take $(h, \id_{\{l\}})$ such that $\dom(h(l))=\dom(f(i))$ (i.e.
$h(l)=f(i)$). The elements $z_1=(f,a)\tau_j$ and $z_2=\tau_i
(h,\id_{l})$ are in $R'$, since $z_1 \rc z_2$, then $z_1=z_2$.
Applying to this equality  $\vph$, we get $(g,a)\tau_j = \tau_i
\vph\big((h,\id_{\set{l}})\big)$, which implies $g(i) =
\theta^{-1}_0 h(l) \theta_0^{\vphantom{-1}} = \theta^{-1}_0 f(i)
\theta_0^{\vphantom{-1}}$.
\end{proof}
\begin{remark}
In the terms of semigroup of transformations of rooted trees this
theorem states that every isomorphism of $\rc$- ($\lc$-)
cross-sections of semigroup of partial ``rooted'' automorphisms of a
rooted regular two-level tree $T$ is a conjugacy.
\end{remark}
\begin{remark}
Because of recursive definition of partial wreath product, one can
ge\-neralize this Theorem for $\rc$-($\lc$-) cross-sections of
semigroup $\is{}_{n_k}\pwr \is_{n_{k-1}}\pwr \ldots \pwr \is_{n_1}$.
\end{remark}
Let $R= R(\overrightarrow{M_1}, \overrightarrow{M_2}, \ldots,
\overrightarrow{M_s})$ be an $\rc$-cross-section of $\isd$. The
vector $(u_1, \ldots, u_n)$, where $u_k=\abs{\{i \med
\abs{M_i}=k\}}$, $1\leq k \leq n$, is called the type of
$\rc$-cross-section $R$. The type of $\lc$-cross-section is defined
in a similar way. It is proven in \cite{GM} that two $\rc$- ($\lc$-)
cross-sections is isomorphic if and only if they have the same type.
It is also shown that the number of non-isomorphic $\rc$-
($\lc$-)cross-sections is equal to $p_n$, where $p_n$ is the number
of decompositions of $n$ into the sum of positive integers, where
the order of summands is not important.
\begin{corollary}
The number of non-isomorphic $\rc$- ($\lc$-) cross-sections of
$\wisdd$ is
$$\sum_{\substack{j_1,j_2,\dots,j_n\ge
0\\j_1+2j_2+\dots+nj_n=n}}\prod_{i=1}^m\binom{p_m+j_i-1}{j_i},$$
where $p_n$ denotes the number of decompositions of $n$ into the sum
of positive integers, where the order of summands is not important.
\end{corollary}
\begin{proof}
Clearly, it is enough to compute the number of non-isomorphic
$\rc$-cross-sections of $\wisdd$. The number of non-isomorphic
$\lc$-cross-sections is the same.

Partition the set of all $\rc$-cross-sections of semigroup $\isdd$
into $p_m$ classes of isomorphic $\rc$-cross-sections and enumerate
them with the integers from 1 to $p_m$.

For a fixed partition $\set{M_1, M_2, \ldots, M_s}$ consider all
$\rc$-cross-sections having the form
\begin{equation} \label{eq:second} (R_1 \pwr
R(M_1))\times \ldots \times (R_k \pwr R(M_s))\end{equation} of semigroup
$\isdd\pwr \isd$ such that $\rc$-cross-section
$R_1=R(\overrightarrow{M_1}, \overrightarrow{M_2}, \ldots,
\overrightarrow{M_s})$. Define as above $j_k=\abs{\set{i  \med
\abs{M_i}=k}}$.

To each $\rc$-cross-section assign a sequence of number pairs
$$\big((|M_1|,i_1), \ldots, (|M_s|,i_s)  \big),\quad
i_l\in\set{1,2,\ldots, p_m},$$
 where $i_l$ is the number of the equivalence class of $R_i\sbs\isdd$. Two $\rc$-cross-sections of the form (\ref{eq:second}) are isomorphic iff their corresponding
sequences are equal up to the permutation of elements. Indeed, we
know from Theorem \ref{thm:perhapsthelast} that any isomorphism of
$\rc$-cross-sections is generated by a ``tree isomorphism''
$\Theta$. Moreover, it follows from the proof of this theorem that
each cofactor $R_k  \pwr R(M_s)$ is mapped to a similar cofactor
$R'_k  \pwr R(M'_k)$ by $\Theta$, and $R'_k \simeq R_k$,
$\abs{M'_k}=\abs{M_s}$. On the other hand, if $M_j=M'_{\sigma(j)}$
and for each $j=1, \ldots, k$  \quad $\psi_j:R'_j \rightarrow
R_{\sigma(j)}$ is an isomorphism, then isomorphism between
$\rc$-cross-sections $(R_1 \pwr R(M_1))\times \ldots \times (R_s
\pwr R(M_s))$ and $(R'_1 \pwr R(M'_1))\times \ldots \times (R'_s
\pwr R(M'_s))$ is established by the map
$$(f_1,a_1)\times\ldots \times (f_s,a_s) \mapsto \big(\psi_1(f_{\sigma(1)}),a_{\sigma(1)}\big)\times \ldots \times \big(\psi_1(f_{\sigma(s)}),a_{\sigma(s)}\big).$$

So, among permutationally-equivalent sequences we can choose a
``canonical'' representation, for instance, we can arrange as
$$\left((1,i_{11}), \ldots,
(1,i_{1j_1}),(2,i_{21}),\ldots,(2,i_{2j_2})  \ldots, \right)$$
$$1\leq i_{11}\leq \ldots \leq i_{1j_1} \leq p_m,\  1\leq i_{21}
\leq \ldots, i_{2j_2} \leq p_m,\  \ldots$$ Thus, we have to compute
the number of representatives in order to find the number of
non-isomorphic $\rc$-cross-sections. To get their number we have  to find the number  of non-decreasing functions from $\set{1,2, \ldots, j_l}$
to $\set{1,2,\ldots, p_m}$.

The number of non-decreasing functions from $\set{1,2, \ldots, j_l}$
to $\set{1,2,\ldots, p_m}$ is equal to the number of solutions of
the equation
$$x_1+\ldots+x_{p_m}=j_l,$$ that is $\binom{p_m+j_l-1}{j_l}$.
So the number of non-isomorphic $\rc$-cross-sections of form
(\ref{eq:second}) is
$$\prod_{i=1}^m\binom{p_m+j_i-1}{j_i}.$$

Summing this over forms (\ref{eq:second}), i.e. over all partitions
of the integer $n$, we get the the number of all non-isomorphic
$\rc$-cross-section of the semigroup $\isdd\pwr \isd$:
$$\sum_{\substack{j_1,j_2,\dots,j_n\ge
0\\j_1+2j_2+\dots+nj_n=n}}\prod_{i=1}^m\binom{p_m+j_i-1}{j_i}.$$
\end{proof}


\begin{thebibliography}{0}

\bibitem{CoReilly} D. F. Cowan, N. R. Reilly,
Partial cross-sections of symmetric inverse semigroups. \emph{Int.
J. Algebra Comput.} \textbf{3}(1995), 259--287.
\bibitem{GM1}
O. Ganyushkin, V. Mazorchuk, \emph{Classical Finite Transformation
Semigroup. An Inroduction}. Series: Algebra and Apllications,
Vol.~9. Springer, 2009.
\bibitem{GM}
O. Ganyushkin, V. Mazorchuk, $\lc$- and $\rc$-Cross-Sections in
$\isd$. \emph{Communications in Algebra}. \textbf{9}(2003),
4507--4523.

\bibitem{eilenberg} S. Eilenberg. \emph{Automata, languages and machines}. Vol.~B. Academic Press, 1976.

\bibitem{comb} Ye. Kochubinska, Combinatorics of partial wreath power of finite inverse
symmetric semigroup $\mathcal{IS}_d$. \emph{Algebra Discrete Math}.
\textbf{1}(2007), 49--60.

\bibitem{endm} Ye. Kochubinska, On cross-sections of partial wreath
product of inverse semigroups. \emph{Electron. Notes Discrete Math}.
\textbf{28}\ (2007), 379--386.

\bibitem{Meldrum} J. P. D. Meldrum, \emph{Wreath products of groups and semigroups}.
Pitman Monographs and Surveys in Pure and Applied Mathematics,
vol.~74. Harlow, Essex: Longman Group Ltd., 1995.

\bibitem{vasia1} V. Pyekhtyeryev, $\hc$- and $\rc$-cross-sections of the full finite semigroup
$T_n$. \emph{Algebra Discrete Math.} \textbf{3}(2003), 82-88.
\bibitem{vasia2} V. Pyekhtyeryev, $\rc$-cross-sections of the
semigroup ${\mathcal{T}_X}$. \emph{Mat. Stud.} \textbf{2}(2004),
133--139.

\bibitem{renner} Lex E. Renner, Analogue of the Bruhat decomposition for
algebraic monoids. II: The length function and the trichotomy.
\emph{J. Algebra}. \textbf{2}(1995), 697--714.

\end{thebibliography}
\end{document}